\documentclass[11pt]{amsart}
\usepackage{amsmath, amssymb, amsthm,color}
\usepackage{amssymb,longtable}
\usepackage{amsfonts}
\usepackage{setspace}
\usepackage{amsmath} 
\usepackage{hyperref} 
\usepackage{graphicx}
\usepackage{latexsym}
\usepackage{array}
\usepackage{amssymb}
\usepackage [autostyle, english = american]{csquotes}
\MakeOuterQuote{"}

\theoremstyle{plain}

\numberwithin{equation}{section}
\newtheorem{thm}{Theorem}[section]
\newtheorem{prop}[thm]{Proposition}
\newtheorem{lemm}[thm]{Lemma}
\newtheorem{cor}[thm]{Corollary}

\theoremstyle{remark}
\newtheorem{rema}[thm]{Remark}

\renewcommand{\div}{{\rm div}}

\newcommand{\bq}{\begin{equation}}
\newcommand{\eq}{\end{equation}}

\def\Xint#1{\mathchoice
{\XXint\displaystyle\textstyle{#1}}%
{\XXint\textstyle\scriptstyle{#1}}%
{\XXint\scriptstyle\scriptscriptstyle{#1}}%
{\XXint\scriptscriptstyle\scriptscriptstyle{#1}}%
\!\int}
\def\XXint#1#2#3{{\setbox0=\hbox{$#1{#2#3}{\int}$ }
\vcenter{\hbox{$#2#3$ }}\kern-.6\wd0}}

\def\dashint{\Xint-}

\begin{document}
\title[Asymptotic Stability for the invscid IPM equation]{On the Asymptotic Stabilty of Stationary solutions of the inviscid incompressible porous medium equation}
\author{Tarek M. Elgindi}
\date{\today}
\maketitle

\begin{abstract}
We study the stability of stationary solutions of the 2D inviscid incompressible porous medium equation (IPM). We show that solutions which are near certain stable stationary solutions must converge as $t\rightarrow\infty$ to a stationary solution of the IPM equation. It turns out that linearizing the IPM equation about certain stable stationary solutions gives a non-local partial damping mechanism. On the torus, the linearized problem has a very large set of stationary (undamped) modes. This makes the problem of long-time behavior more difficult since there is the possibility of a cascading non-linear growth along the stationary modes of the linearized problem. We solve this by, more or less, doing a second linearization around the undamped modes, exploiting a special non-linear structure, and showing that the stationary modes can be controlled.

\end{abstract}
\setstretch{1.3}

\section{Introduction}

Incompressible fluid equations have received much attention from the PDE community in recent years due to the mathematical challenges and many interesting phenomena that they present. Depending upon the specific physical situation that a given fluid equation models, we find vastly different mathematical objects arising. This is particularly the case when studying questions related to the long-time behavior of solutions. In recent years, researchers have discovered numerous interesting phenomena such as the existence of a wide array of stationary states, solutions which are periodic (see \cite{HmidiMateu} or \cite{Cordoba}) or quasi-periodic in time (see \cite{CF} or \cite{ElgindiJeong}), the existence of solutions which experience rapid growth leading to singularity formation either in finite \cite{KRYZ} or infinite time \cite{KiselevSverak}, the existence of solutions whose long-time behavior is determined entirely by some linear or dispersive effect \cite{GMS}, the existence of solutions whose long-time behavior are determined by some non-linear mixing effect (inviscid damping)(see, for example \cite{BedrossianMasmoudi}, \cite{VillaniMouhot}, or \cite{BMM}), as well as the existence of solutions which simply decay to 0 due to coercive dissipative mechanisms \cite{MR1925398}. 

The purpose of this work is to investigate a phenomenon which does not seem to have received much attention in the fluids community until recently: partial dissipation. Roughly speaking, a \emph{fully} dissipative system is one where all "movement" is damped by some mechanism such as diffusion or drag. A \emph{partially} dissipative system, on the other hand, is one where only certain types of motion are damped. It is easy to see that both of these situations can arise rather naturally from different physical scenarios. For example, in a physical system where there is gravity and stratification, vertical movement may be penalized while horizontal movement is not.
\subsection{The inviscid IPM equation}

We will study fluids which are stratified by density in the absence of diffusion; such fluids can be modeled by the the inviscid incompressible porous medium equation:

\begin{equation}
\frac{\mu}{\kappa}u=-\nabla p-(0,g\rho),
\end{equation}
\begin{equation}
\partial_t\rho+u\cdot\nabla\rho=0,
\end{equation}
\begin{equation} \div(v)=0\end{equation}
where $u$ is the velocity field of the fluid, $p$ is the pressure, $\mu$ is the dynamic viscosity, $\kappa$ is the permeability of the isotropic medium, $\rho$ is the liquid density and $g$ is the gravitational acceleration. 
When these equations are studied on a bounded domain, we assume that $u$ satisfies the no-slip boundary condition:
$$u\cdot n=0,$$
on the boundary of the domain where $n$ is the normal to the boundary.
Our goal here is to study the (non-linear) stability of exact solutions to this system. For simplicity, let's take $\mu=\kappa$ and $g=-1$. 
We note that this system can also be written as follows:
\begin{equation} \partial_t\rho+ u\cdot\nabla \rho=0\end{equation}
\begin{equation} u=R_1 R^\perp \rho \end{equation}
where $R^\perp=(R_2,-R_1)$ and $R_1, R_2$ are the Riesz transforms and this is exactly the same as the surface quasi-geostrophic system except that $u=R^\perp \rho$ in that case (see \cite{Resnick}). 

We will show that the stratification inherent in the model serves as a stabilizing mechanism for solutions which are in a small (Sobolev) neighborhood of certain stable steady states. Indeed, one can imagine that a fluid with density that is proportional to depth (i.e. that the density of the fluid increases the deeper you go into the fluid)
is, in some sense, ``stable.'' On the other hand, if the density is inversely proportional to depth, then one would imagine that such a scenario is unstable--indeed this is where one sees the so-called Rayleigh-Darcy convection or Rayleigh-Benard convection even in the presence of diffusion \cite{MR1925398}.
We are concerned mainly with the stable case. While it is clear that \emph{with} diffusion in the equation for the density, small solutions will eventually decay to zero, this is not clear \emph{without} diffusion. The question we wish to ask here is whether one is able to establish non-linear stability results for (1.1)-(1.3). We will show that this is indeed the case.
In fact, we will be able to prove that smooth perturbations of the stationary solution $\rho(x,y):=y$ are stable for all time in Sobolev spaces. 

We will do this in two settings which are fundamentally different: on the whole space $\mathbb{R}^2$ and on the two dimensional torus $\mathbb{T}^2$. On $\mathbb{R}^2$ the corresponding linearized problem has a damping mechanism and we are able to show that perturbations of a constant gradient density go to 0 as time goes to infinity. This is done despite the very weak damping which the linear problem affords us by using a special structure in the non-linear term. The problem on $\mathbb{T}^2$ is different. In that case the linearized problem does not damp all Fourier modes and is partially dissipative. In fact, there is a large set of undamped modes which preclude the possibility of the perturbation vanishing as time goes to infinity. Despite this, we will show that the full solution settles (in the long time limit) on another stratified stationary solution that is determined by the non-linear evolution. 

\subsection{Coercive vs. non-coercive dissipation}
Before going into the details of the IPM equation, we wish to discuss the general picture of the \emph{type} of problem we are looking at and the differences between partially dissipative and fully dissipative systems. Consider the following abstract equation:
$$  \partial_t f +N(f)=L(f) $$
Here we assume that $L$ is a negative linear operator in the sense that $$(L(f),f)= -|A f|^{\frac{1}{2}}$$ for some linear operator $A$ and $N$ is a non-linear operator in the sense that in some nice function spaces $X$ and $Y$ (say $H^s$ and $H^{s+1}$ for example) $N$ satisfies:
$$|N(f)|_{X}\lesssim |f|_{Y}^2.$$  
A simple example of a fully dissipative system is the case where $N(f)=f^2$ and $L(f)=-f$. In this case we have:

$$\partial_t f=f^2-f.$$ Now, in the setting of small initial data, where the sup-norm of the initial profile $f_0$ is strictly smaller than 1, it is clear that any initial profile will decay to $0$ exponentially fast as $t\rightarrow\infty.$ A simple example of an \emph{partially} dissipative system is the following simple model equation posed on $[-\pi,\pi)$:
$$\partial_t f=f^2 -(f- \tilde f),$$ where $\tilde f:=\frac{1}{2\pi}\int_{-\pi}^\pi f(x,t)dx.$
One notices that $L(f)=f-\tilde f$ is non-negative on $L^2(-\pi,\pi)$:
$$(f-\tilde f, f)_{L^2}=|f-\tilde f|_{L^2}^2\geq 0.$$
However, despite the linearized problem being stable, it is clear that if $f_0\equiv \delta$ for any constant $\delta>0,$ then $f$ actually grows to infinity in finite time! This is to say that the linear problem, while damping all Fourier modes except the constant mode, was not strong enough to stop the non-linear term from causing infinite growth no matter how small the data is. On the other hand, a second instructive example is the following one:
$$\partial_t f=f\partial_x f - (f-\tilde f).$$ For this system, in fact, one can prove that if the quantity $|\partial_x f|_{L^\infty} $ is initially small enough then there exists a global strong solution for which $\partial_x f$ remains small for all time. This can be seen in the following way:
$$\partial_t \partial_x f=(\partial_x f)^2+f\partial_{xx}f -\partial_x f.$$ Hence, by a maximum principle argument,  if $|\partial_x f|_{L^\infty}\leq 1$ initially, then $|\partial_x f|_{L^\infty}\leq 1$ for all time. In fact, it is easy to see that $A(t):=|f|_{L^\infty}(t)$ actually satisfies the ODE: $$\frac{d}{dt}A=A^2-A$$ so that if  $|\partial_xf|_{L^\infty}<1$ initially, it must decay exponentially. Furthermore, it is easy to see that the average of $f$ is constant in time. Hence, one sees that $f(t,\cdot)\rightarrow \tilde f_0$ exponentially as $t\rightarrow\infty$ strongly in $W^{1,\infty}$. The key difference between this example and the previous one is that the stationary mode--the average of $f$--could undergo a non-linear growth in the first example while, in the second example, the linearly stationary mode cannot "produce itself" through the non-linearity.    

We see, even at this simple level, that the \emph{precise} interaction between the linear partial damping term and the non-linear term is important. We now come back to our more abstract setting:
$$  \partial_t f +N(f)=L(f) $$ with $$(Lf,f)=-|Af|^2$$ where the inner product and norm are taken in a suitable function space.
We have seen that unless $A$ is coercive in the sense that $|Af|\geq c |f|$ it is generally difficult to prove that 0 is asymptotically stable\footnote{It is important to note here that there is a difference between what we are describing here and hypocoercivity in that in the setting of hypocoercivity there are two non-commuting linear operators: one partially dissipative operator as we are describing and another which is skew symmetric. In the setting of hypocoercivity the skew symmetric operator 'helps' the partially dissipative operator and the effect of the two becomes in some sense fully dissipative. This is not the case here where we only have the partially dissipative operator by itself. }. Nonetheless, it is possible that for some non-linear problems we can still prove stability in the following way. 

Let us suppose that we are studying this equation on the periodic domain $(-\pi,\pi)$. Let us further suppose that $\mathbb{Z}$ (which represents the Fourier modes) can be decomposed into two disjoint sets $D$ and $S$, the dissipative set and the stationary set: $$\mathbb{Z}=D \cup S$$ such that $$L(e^{inx})=0$$ for $n\in S$ and $$L(e^{inx})=-e^{inx}$$ for $n\in D$. 

In this case, it seems that if we solve $$\partial_t f + N(f) = L(f)$$ that small solutions $f$ may not decay to 0 but maybe they can be decomposed into $f_D$ and $f_S$ (a dissipative part and a stationary part) such that $f_D$ goes to 0 and $f_S$ merely remains small. Being able to do this effectively will depend greatly upon the nature of $D$ and $S$ (i.e. the structure of the operator $L$) as well as the nature of the nonlinearity $N$. Determining general conditions on $N$ and $L$ to ensure that the solution $f$ can be controlled for all time seems to the author to be a fundamental problem. In this paper, we encounter precisely this setting and we deal with it by studying the linearization of $N$ around the corresponding set $S$. In our setting, one point which is immensely important is that $$N(f_S)=0.$$ 

Getting general conditions on $N$ and $L$ to ensure that solutions remain bounded for all time seems to be a challenge. However, the intuition one gathers from this work is that there should be a condition on the linearization of $N$ around the elements of the stationary set $S$.

In the coming subsections we outline the main results of this work, which is the asymptotic stability of the stratified state $\rho(x,y)\equiv y$ both for perturbations which decay on $\mathbb{R}^2$ and perturbations which are periodic. We begin by observing that the full non-linear system has a very simple class of steady states when the equation is considered on a closed domain. 

\subsection{The steady states}When studying fluid equations, it is often helpful to have a good understanding of the exact steady states of the system.
The kinds of exact solutions we are interested in are: $\rho=f(y)$, $u=0$, and $p=0$. It is trivial to check that these are stationary solutions for any $C^1$ function, $f$. 
In fact, under mild assumptions, these are the only stationary solutions to the equation when studied on a bounded domain. 
\begin{lemm}
Let $\rho$ be a $C^{1}$ stationary solution of the inviscid IPM equation on a bounded domain $\Omega$. Then, $\rho$ is a function of $y$ only and, in particular, $u\equiv 0$. 
\end{lemm}

\begin{proof}

We know that $$u\cdot\nabla\rho=0.$$
This means that $$\int u\cdot\nabla\rho y=0.$$ 
Due to the no-slip condition on $u$ in the bounded domain case and using the divergence-free condition we see:
$$\int u\cdot\nabla y \rho=0.$$
This implies, on the one hand, that:
$$\int u_2 \rho=0$$
On the other hand, if we take the equation $$u+\nabla p=(0,\rho)$$ and dot with $u$ we see (after an integration by parts and using the boundary conditions)
$$\int u_2\rho=\int |u|^2.$$
Thus, $u\equiv 0$. As a result, $\rho_x\equiv 0$ and the lemma is proved.
\end{proof}

An easy corollary of the proof of Lemma 1.1 is that, in the $\mathbb{R}^2$ or bounded domain case, solutions $\rho(x,y,t)$ gravitate towards becoming a function of $y$ only:  
\begin{cor}
If $\rho$ is a smooth solution (1.4)-(1.5), with $\rho$ decaying sufficiently fast, $$\partial_t \int \rho(x,y) y dxdy=-\int |u(x,y)|^2 dxdy.$$
Hence, $$\partial_t \int |\rho(x,y)-y|^2dxdy=-2\int|u(x,y)|^2.$$ 
\end{cor}

Although, as far as active scalar equations go, (1.4)-(1.5) shares many similarities with the surface quasi-geostrophic equation and the 2D Euler equations, the IPM system has a very simple structure of stationary solutions. It turns out, in addition, that the linearized IPM equation around any one of these steady states has 
a special structure which allows us to deduce results on the long time behavior of perturbations of stationary solutions in a relatively easy fashion. In fact, this special structure, at least in the $\mathbb{R}^2$ case, is already in Corollary 1.2.  

Now, as was shown in Lemma 1.1, $\rho=f(y)$ and $u=0$ are stationary solutions of this system. Now suppose that we perturb these stationary solutions a little bit: $\rho=f(y)+\tilde{\rho}$ and $u=\tilde{u}.$ Then we see:
$$\tilde{u}=-\nabla p +(0,\tilde{\rho})$$
and 
$$\partial_t \tilde\rho+(\tilde u)\cdot\nabla (f(y)+\tilde\rho)=0.$$
Now we rewrite $\tilde u$ as $u$ and $\tilde\rho$ as $\rho$ and simplify:
\begin{equation}{u}=-\nabla p +(0,{\rho})\end{equation}
$$\partial_t \rho +u_2f'(y)+v\cdot\nabla\rho=0$$

It is clear that under the no-slip boundary condition ($u\cdot n=0$) and under the assumption that the domain is simply connected, we may pass to the stream function formulation where we define $u=\nabla^\perp \psi$:
$$\Delta \psi =-\partial_x\rho$$ and $\psi=0$ on the boundary of the domain. 
This implies that $u_2=-R\rho$ where $R=\partial_{xx}(-\Delta)^{-1}$. Here, $(-\Delta)^{-1}$ is the inverse of the Dirichlet Laplacian. 
Thus, our equation reads:
\begin{equation}\partial_t \rho -f'(y)R\rho+u\cdot\nabla\rho=0\end{equation}

What is interesting about this equation is that $R$ is a negative operator so we get a mild dissipation effect.  
This structure will allow us to prove stability. 

\begin{thm} (Main Result on $\mathbb{R}^2$)
 
Let $\Omega(y):=y$.There exists $\epsilon_0>0$ such that if we solve IPM with initial data $\rho_0=\tilde\rho_0+\Omega$ with $|\tilde\rho_0|_{W^{4,1}}+|\tilde\rho_0|_{H^s}\leq \epsilon\leq \epsilon_0$, $s\geq 20$ then the solution $\rho$ satisfies the following: 

\begin{enumerate}
\item \qquad\qquad\qquad\qquad $ |\rho(t)-\Omega|_{H^3}\lesssim \frac{\epsilon}{t^{1/4}} \,\,\, \forall \, t>0,$
\item \qquad\qquad\qquad\qquad $ |u_1|_{H^3}\lesssim \frac{\epsilon}{t^{3/4}},$
\item \qquad\qquad\qquad\qquad $|u_2|_{H^3}\lesssim \frac{\epsilon}{t^{5/4}},$
\end{enumerate}

where $u=R_1 R^\perp {\rho}.$

\end{thm}

\begin{thm} (Main Result on $\mathbb{T}^2$)
 
The stationary solution $\Omega$ of the IPM equation is asymptotically stable in $H^s(\mathbb{T}^2)$ for $s\geq 20.$ In other words, there exists $\epsilon_0>0$ such that if we solve IPM with initial data $\rho_0:=\tilde\rho_0+\Omega$ with $|\tilde\rho_0|_{H^s}\leq \epsilon\leq \epsilon_0$ then the solution $\rho$ satisfies the following: 

\begin{enumerate}
\item \qquad\qquad\qquad\qquad$ |\rho(t)-\Omega|_{H^{20}}\leq 2\epsilon \,\,\, \forall \, t>0$
\item  \qquad\qquad\qquad\qquad$ |u|_{H^3}\lesssim \epsilon t^{-2.5}$
\end{enumerate}
\end{thm}

\begin{rema}
We note here that $\Omega$ does not belong to $L^2(\mathbb{R}^2)$. However, if we perturb $\Omega$ by 
an $H^s$ function the perturbation will remain $H^s$ for all time (unless the solution blows up in finite time). 
Similarly, $\Omega$ is not periodic but we may perturb it by a periodic function and once more the perturbation will remain periodic. Note also that in the $\mathbb{R}^2$ case, $\rho$ itself will decay (though mildly); however, in the $\mathbb{T}^2$ case, $\rho$ will not decay.  Note that an easy consequence of  Theorem 1.4 is that the $x-$ derivatives of $\rho$ decay algecraically whereas the $y$ derivatives need not decay at all. Indeed, if we perturb the stationary solution by a function of $y$ only then there should be no decay! As it will be clear later, the proofs of Theorems 1.3 and 1.4 are actually fundamentally different. 
\end{rema}

\subsection{Comparison with other results}

The idea of taking a non-linear equation where global well-posedness is unknown and proving global well-posedness near a stationary solution of the equation is not new. However, it seems that operators which are partially dissipative are less studied than fully dissipative ones. 
Very recently, this was done for the MHD equation by Lin and Zhang \cite{LinZhang} and has also been done for complex fluids and in other contexts as well (\cite{CheminMasmoudi},\cite{Wu1}). 

In fact, the linear propagator $$\partial_t g =R(g)$$ was already present in the work \cite{LinZhang}. The authors only studied the problem on $\mathbb{R}^2$ and relied on methods similar to the ones in Section 3 of this paper, except that they chose to prove existence in anisotropic spaces due to the fact that $R$ is anisotropic.  This is also possible in our setting but we chose not to go that route to keep the exposition simple. Furthermore, we emphasize that for problems related to partial damping it is important to consider the problem on periodic domains since there is where partial damping is really different from coercive damping. Indeed, we will show that when the operator $R(g)$ is considered on $\mathbb{R}^2$ it is still "fully" damping but it is much weaker than damping by $-g$.

In the direction of global existence and uniqueness for supercritical active scalar equations, one notable result in the class of vortex-patch type solutions is that of Hmidi and Hassainia \cite{2014arXiv1402.6499H}. They prove global existence and uniqueness of a certain kind of ``periodically-rotating'' vortex patch
for a class of supercritical active scalar equations. This was subsequently improved to a broader class of models (including the SQG equation) by Cordoba et al. \cite{Cordoba}. 
In the class of weaker solutions, Isett and Vicol \cite{IsettVicol} were able to use convex integration to construct global weak solutions to the IPM equation which are of class $C^\frac{1}{9}_{t,x}.$ Unfortunately, as is established by Isett and Vicol, 
these solutions are highly non-unique. A slightly different class of results were recently attained by a number of authors (\cite{Muskat1},\cite{Muskat2},\cite{Muskat3}) on the Muskat problem which can be seen as the ``free boundary problem'' for the IPM equations. 

We close by mentioning that in \cite{IsettVicol},\cite{FRV}, and \cite{FRV2}, \cite{FGSV} the authors indicate a difference between active scalars where the operator relating the velocity field and the advected quantity ($u$ and $\rho$ in our case), has an even Fourier symbol or an odd Fourier symbol. We are taking advantage of the fact that, for the IPM equation, this symbol is even and one component of it has a sign. Such a thing can never happen if the symbol is odd;
however, in \cite{ElgindiWidmayer} the authors establish certain dispersive properties of equations of the form $f_t =R_1 (f)$ which can also act as a stabilizing force. It is possible that using this sort of dispersion, one can say something 
about stationary solutions for active scalar equations with odd symbol.   

\subsection{The ideas behind the proofs}

The proofs of Theorems 1.3 and 1.4 are of a very different nature. In the whole space case (Theorem 1.3), if we solve the linearized problem $\rho$ itself decays at the rate of $t^{-\frac{1}{4}}$, $u_1$ decays at a rate of $t^{-3/4}$, and $u_2$ decays at a rate of $t^{-5/4}$. It is somewhat inconcievable that such slow decay rates can control a \emph{general} quadratic non-linearity. However, it turns out that if we analyze the nonlinear term:

$$u\cdot\nabla\rho= (-R_1 R_2 \rho, R_1^2\rho)\cdot\nabla \rho,$$ then we will notice that each term of the nonlinearity contains \emph{two} x-derivatives which allows us to prove that the energy will be controlled so long as $\nabla u_2$ is controlled, and as stated above $u_2$ decays like $t^{-5/4}$ this is done in subsection 3.1. So we can control the energy so long as we can bootstrap a decay of $t^{-5/4}.$ Bootstrapping the decay of $u_2$ is non-trivial due to the fact that $\rho$ itself decays very slowly. Nevertheless, in subsections 3.2 and 3.3 we are able to prove (with a great loss of derivatives) that so long as $\rho$ is bounded in a high energy space $u_2$ decays like $t^{-5/4}$.

\vspace{5mm}

The result on the torus (Theorem 1.4) is very different for the main reason that $\rho$ itself does not decay. In fact, as is given in the appendix, there are very simple active scalar equations (with a different nonlinearity from ours) for which the linearized equation is not strong enough to stop the formation of finite time singularities. Hence, it is necessary to use some of the finer properties of the non-linearity to deduce long-time stability.
The main difficulty in proving the stability on the torus is that the linearized equation has a large set of stationary modes. On the other hand, the linearized equation gives very good decay properties for the velocity field $u$ \emph{so long as} we are willing to lose derivatives (Proposition 2.7). A loss of derivatives in the linearized decay estimate along with the fact that there are non-decaying modes in the equation makes it nearly impossible to propagate decay unless we are willing to work in super-smooth spaces (Gevrey-Sobolev spaces), and even then unless there are suitable cancellations in the non-linearity we may lose stability (see appendix). Indeed, the non-decaying mode will actually introduce a \emph{new} linear term into the equation which could, potentially, change the decay properties of the linearized problem.  We get around this problem by proving that the decay properties of the linear semi-group $e^{R_1^2t}$ are actually stable in a certain sense.

More specifically, if we consider the solution $\rho(x,y,t)$, we break $\rho$ up:
$$\rho=\tilde\rho+\bar\rho$$ with $\tilde\rho$ being the non-decaying part of $\rho$ and $\bar\rho$ being the decaying part of $\rho$. Then we consider the nonlinearity:
$$u\cdot\nabla\rho=(\tilde u + \bar u)\cdot\nabla (\tilde\rho+\bar\rho)= \tilde u \cdot\nabla \tilde \rho + \bar u\cdot\nabla \tilde \rho+\tilde u\cdot\nabla \bar\rho +\bar u\cdot\nabla \bar \rho. $$
Now, it happens that $\tilde u\equiv 0$. Hence the non-linearity collapses into:
$$u\cdot\nabla \rho = \bar u \cdot\nabla\tilde\rho + \bar u \cdot\nabla \bar\rho.$$

Since $\bar\rho$ is decaying, the term $\bar u\cdot\nabla\bar\rho$ should be very small and should be controllable. The term $\bar u \cdot\nabla\tilde\rho$, however, acts like a second linear operator since $\tilde\rho$ is not decaying. It is conceivable that in some problems, this extra linear operator could compete with the damping coming from the linear term. By showing that, as we stated above, the decay mechanism is "stable" with respect to the sort of perturbations which this second linear operator introduces, we are able to keep the decay mechanism and close a decay estimate for $\bar\rho$ and show that $\tilde\rho$, while not decaying, converges as $t\rightarrow\infty$.

\subsection{Organization of the Paper}

In the next section we will study some properties of the linear equation $$\partial_t\rho=R\rho$$ on $\mathbb{R}^2$.  
In Section 3 we will prove Theorem 1.3. In sections 4 and 5 we will prove the linear results and non-linear results (Theorem 1.4) on the torus. 

\section{Linearized equation and Linearized Decay on $\mathbb{R}^2$}

In this section we prove $L^2$ decay estimates for the linear equation:

$$\partial_t \rho = R_1^2 \rho.$$

By taking the Fourier transform, we see:   $$\hat\rho(t,\xi)=e^{- \frac{\xi_1^2}{\xi_1^2+\xi_2^2}t}\hat\rho_0 (\xi_1,\xi_2).$$
Thus, $$|\rho|_{L^2}^2 = |\hat\rho|_{L^2}^2 = \int_{\mathbb{R}^2}e^{- \frac{\xi_1^2}{\xi_1^2+\xi_2^2}t}|\hat\rho_0 (\xi_1,\xi_2)|^2d\xi_1d\xi_2.$$

Now let's transform this integral into an integral in polar coordinates.

Then we have:

$$|\hat\rho|_{L^2} = \int_{\mathbb{R}^2}e^{- \frac{\xi_1^2}{\xi_1^2+\xi_2^2}t}|\hat\rho_0 (\xi_1,\xi_2)|^2d\xi_1d\xi_2=\int_{0}^{2\pi}\int_{0}^\infty e^{-cos^2(\theta)t} |\hat\rho_0(\theta,r)|^2 rdr d\theta$$
$$=\int_0^{2\pi} e^{-cos^2(\theta)t} G_0(\theta) d\theta.$$

We now use the following calculus lemma.

\begin{lemm}
$$\int_{0}^{2\pi} |cos(\theta)|^{k} e^{-cos^2(\theta)t}d\theta \approx c_k t^{-(1+k)/2}\,\,\,\, \text{as}\,\,\,\, t\rightarrow \infty$$

\end{lemm}

The proof of this lemma is simple and basically comes down to localizing to the region where $|cos(\theta)|$ is small which consists of  the region $|\theta-\frac{\pi}{2}|<<1$ and the region $|\theta-\frac{3\pi}{2}|<<1$. A suitable transformation yields that the integral above is approximately $$\int_{-1}^{1} |x|^k e^{-x^2 t} dx\approx \frac{c_k}{\sqrt{t^{k+1}}}$$ for some constant $c_k$. 

Now, with Lemma 2.3 at hand, we see that $$|\hat\rho|_{L^2} \lesssim \frac{1}{\sqrt{t}} |G_0|_{L^\infty([0,2\pi])}.$$

On the other hand, $$G_0(\theta) = \int_0^\infty |\hat{\rho}(\theta,r)|^2rdr.$$

Assuming $\rho_0 \in W^{1+\delta,1}$ we have that:

$$|\hat{\rho_0}(\theta,r)|\leq \frac{|\rho_0|_{W^{1+\delta}}}{r^{1+\delta}+1}.\,\,\, $$ 

Hence, $$|G|_{L^\infty} \lesssim |\rho_0|_{W^{1+\delta,1}}^2$$

Thus we arrive that the following proposition:

\begin{prop}
Let $\rho_0 \in W^{1+\delta,1}$ for some $\delta>0$. 
If

$$\hat\rho(t,\xi)=e^{- \frac{\xi_1^2}{\xi_1^2+\xi_2^2}t}\hat\rho_0 (\xi_1,\xi_2),$$

then, 

\begin{equation} |\rho|_{L^2} \lesssim \frac{|\rho_0|_{W^{1+\delta,1}}}{(t+1)^{1/4}}.\end{equation}

Moreover, \begin{equation} |R_1\rho|_{L^2} \lesssim \frac{|\rho_0|_{W^{1+\delta,1}}}{(t+1)^{3/4}},\end{equation}

and 

\begin{equation} |R_1^2 \rho|_{L^2} \lesssim \frac{|\rho_0|_{W^{1+\delta,1}}}{(t+1)^{5/4}}.\end{equation}

\end{prop}

Unfortunately, estimates (2.1)-(2.3) have losses of derivatives in them; it would be optimal to be able to get an estimate on $\rho$ in $L^2$ in a way which doesn't lose derivatives. 
We can, at the expense of time decay, prove the following:

\begin{prop} 

Let $\rho_0\in L^2.$

If $\hat\rho(\xi,t)= e^{-\frac{\xi_1^2}{|\xi|^2}t}\hat\rho_0(\xi),$ then, 

\begin{equation} |\rho|_{L^2} \leq |\rho_0|_{L^2} \end{equation} 

\begin{equation} |R_1\rho|_{L^2} \leq \frac{1}{\sqrt{t+1}}|\rho_0|_{L^2}\end{equation} 

\begin{equation} |R_1^2\rho|_{L^2} \leq \frac{1}{{t+1}}|\rho_0|_{L^2} \end{equation}

\end{prop}

\begin{proof}

The proofs of (2.7)-(2.9) are a direct consequence of the following pointwise inequality:

\begin{equation} \Big|e^{-A^2 t} A^k\Big| \leq \frac{C_k}{(t+1)^{k/2}}\end{equation} for $|A|\leq 1.$

We will prove (2.8) only as (2.6)-(2.7) are similar: $$|R_1\rho|_{L^2}^2= \int e^{-\frac{\xi_1^2}{|\xi|^2}t} \Big |\frac{\xi_1}{|\xi|}\Big|^4 |\hat{\rho}(\xi_1,\xi_2)|^2 d\xi_1 d\xi_2.$$

Now (2.9) follows from inequality (2.10) applied with $A=\frac{\xi_1}{|\xi|}.$

\end{proof}

We will need both Proposition 2.2 and Proposition 2.3 to prove Theorem 1.3.

\subsection{A Basic Lemma}

The following basic lemma will be needed throughout the paper. 

\begin{lemm}

Let $\delta>0$ and $\eta>0$.

Then,

\begin{equation} \int_{0}^{t}\frac{ds}{(t-s+1)^{\delta}(s+1)^{1+\eta}} \leq \frac{C_{\eta,\delta}}{(t+1)^{\delta}} \end{equation} 

\end{lemm}

\begin{proof} 

Case 1: $\delta\not=1$.

$$ \int_{0}^{t/2}\frac{ds}{(t-s+1)^{\delta}(s+1)^{1+\eta}} \leq \frac{1}{(t/2+1)^\delta}\int_0^{t/2} \frac{ds}{(s+1)^{1+\eta}}= \frac{1}{\eta(t/2+1)^\delta} $$

$$\int_{t/2}^{t}\frac{ds}{(t-s+1)^{\delta}(s+1)^{1+\eta}} \leq\frac{1}{(t/2+1)^{1+\eta}}\int_{t/2}^{t} \frac{ds}{(t-s+1)^{\delta}}$$
$$=\frac{1}{(t/2+1)^{1+\eta}}\int_0^{t/2} \frac{ds}{(s+1)^{\delta}}= \frac{1}{{1-\delta}}\frac{1}{(t/2+1)^{1+\eta}} ((t/2+1)^{1-\delta} -1) $$

$$\lesssim C_\delta \frac{1}{(t/2+1)^{\delta+\eta}}$$

This completes Case 1.

When $\delta=1$ we simply get:

$$\int_{t/2}^{t}\frac{ds}{(t-s+1)^{\delta}(s+1)^{1+\eta}} \leq \frac{Log(t/2+1))}{(t/2+1)^{1+\eta}} \leq \frac{c}{1+t}. $$

This concludes the proof of 2.11.

\end{proof}

\section{The proof of Theorem 1.1}

Our equation reads: \begin{equation}\partial_t \rho +u\cdot\nabla\rho=R_1^2\rho \end{equation}

The goal will be to prove

(A) Decay of $u$ with $u_2$ decaying in an integrable way if $\rho$ remains small. 

(B) Energy estimates which require only integrable decay of $u_2$ to show that $\rho$ remains small. 

We begin by proving special energy estimates which allow us to say that we can prove that the $H^s$ norms of $\rho$ remain small so long as the $\nabla u_2$ decays fast enough. Notice that in most cases with a transport equation, we need fast decay of the gradient the whole velocity field, $\nabla u,$ in order to control the equation. We will use a special structure of the nonlinearity to prove that only decay on $\nabla u_2$ is needed.

\subsection{Energy Estimates}

\begin{lemm} The following estimate holds for $s\geq 4.$

\begin{equation}
\partial_{t}|\rho|_{H^{s}}^2\leq C\Big(|\nabla u_2|_{L^\infty}|\rho|_{H^{s}}^2+|u|_{H^{s}}^2|\rho|_{H^{s}}\Big)-2|u|_{H^{s}}^2.
\end{equation}

\end{lemm}

\begin{rema}
A consequence of Lemma 3.1 is that if we have good (integrable) time decay on $u_2$ and its gradient, then we will be able to prove that $|\rho|_{H^{s}}$ remains uniformly bounded by $2\epsilon$ so long as it starts out of size $\epsilon$ small enough.  We will indeed prove this in the following subsection. 

\end{rema}

\begin{proof}

The usual method of using the Kato-Ponce inequality will only give us $|\rho|_{H^s}^2|u|_{H^s}$ on the right hand side of the energy inequality. We will need to carry out the energy estimates carefully to ensure that we get the desired estimate. 

We first want to make the following simple observations:

(1) $|R_{1}\rho|_{H^s}=|u|_{H^s}.$

(2) $|\partial_{x}\rho|_{H^s}=|R_{1}\Lambda\rho|_{H^s}= |R_1\rho|_{H^{s+1}}=|u|_{H^{s+1}}$

We are interested in $H^s$ estimates so we will focus on first controlling the non-linear term $u\cdot\nabla\rho$.

\emph{Step 1: The non-linear term}

$$(\partial^s(u\cdot\nabla\rho),\partial^s\rho)=\sum_{i=1}^s a_{i,s}\Big(\partial^iu\cdot\nabla\partial^{s-i}\rho,\partial^s\rho\Big).$$

So we must study $$\Big(\partial^iu\cdot\nabla\partial^{s-i}\rho,\partial^s\rho\Big),$$ for $1\leq i\leq s.$

Using observation (1)-(2) above, it actually suffices to consider the term $$\Big(\partial^i\psi_{x}\partial^{s-i}\rho_y,\partial^s\rho\Big),$$ where $u=\nabla^\perp\psi$. Now, we want to distinguish between two kinds of terms, first the case where $i=1$ and then the case where $i\geq 2$. 

\vspace{3mm}
\emph{The case $i=1.$}
\vspace{3mm}

This means that we study $$(\partial u_2\partial^{s-1}\rho_y,\partial^s\rho).$$

This term is bounded by $|\nabla u_2|_{L^\infty} |\rho|_{H^s}^2$. 

\vspace{3mm}
\emph{The case $i\geq 2$.}
\vspace{3mm}

We will study $$(\partial^i\psi_x\partial^{s-i}\rho_y,\partial^s\rho).$$

Upon integrating by parts, we see:

$$(\partial^i\psi_x\partial^{s-i}\rho_y,\partial^s\rho)=-(\partial^{i}\psi\partial^{s-i}\rho_{xy},\partial^s\rho)-(\partial^i\psi\partial^{s-i}\rho_y,\partial^{s}\rho_x)=I+II.$$

Now, it is clear that $$|I|\lesssim |u|_{H^s}^2|\rho|_{H^s}.$$ Moreover, by writing:

$$II=(\partial(\partial^i\psi\partial^{s-i}\rho_y),\partial^{s-1}\rho_x)$$ and noting that $i\leq 2$ we see:

$$|II|\lesssim |u|_{H^s}^2|\rho|_{H^s}.$$

\end{proof}

\subsection{Decay of $|\rho|_{H^5}$ and $|u_2|_{H^{10}}$}

We will now prove the following proposition:

\begin{prop}
Assume that $|\rho|_{H^{20}}\leq 4\epsilon$ on the interval $[0,T]$. Then,  \begin{equation} |\rho|_{H^5}\lesssim \frac{\epsilon}{(t+1)^{\frac{1}{4}}},\end{equation}

and

\begin{equation} |u_2|_{H^{10}}\lesssim \frac{\epsilon}{(t+1)},\end{equation}

for all $t\in [0,T]$. 

\end{prop}

The proof of (3.2) and (3.3) is somewhat delicate because there is a loss of derivatives in the decay estimate (2.1). 

\begin{proof}

Using Duhamel's formula, we can write:

$$\rho(t) = e^{R_1^2 t} \rho_0 + \int_0^t e^{R_1^2(t-s)} (u\cdot\nabla \rho)(s) ds$$

Using (2.4) we see that:

$$|\rho|_{H^5} \leq \frac{C\epsilon}{(t+1)^{1/4}}+\int_0^t \frac{C}{(t-s+1)^{1/4}} |u\cdot\nabla\rho|_{W^{7,1}}.$$

We will need the following estimate:

\emph{Claim:} 

$$|u\cdot\nabla\rho|_{W^{7,1}} \leq C\sqrt{\epsilon} |u_2|_{H^{10}}\sqrt{|\rho|_{H^5}}.$$

\emph{Proof of the Claim:}

$u\cdot \nabla\rho$ consists of two terms:

$$ |u\cdot\nabla\rho|_{W^{7,1}}\leq C |u_1|_{H^7}|\partial_x\rho|_{H^7}+ |u_2|_{H^7}|\partial_y\rho|_{H^7}$$

$$|u_1|_{L^2}^2 = \int R_1R_2\rho R_1R_2\rho=\int R_2R_2\rho R_1R_1\rho\leq |\rho|_{L^2} |R_1^2\rho|_{L^2}=|\rho|_{L^2} |u_2|_{L^2}$$

Hence, $$|u_1|_{H^7} \leq |\rho|_{H^7}^{1/2} |u_2|_{H^7}^{1/2}. $$

By the same token, $$|\partial_x\rho|_{H^7} \leq |\rho|_{H^8}^{1/2} |u_2|_{H^8}^{1/2}. $$

Hence, $$ |u\cdot\nabla\rho|_{W^{7,1}}\leq C |u_2|_{H^8}|\rho|_{H^8}$$

However, due to the well-known interpolation inequalities, $$|\rho|_{H^8}^2 \leq |\rho|_{H^5}|\rho|_{H^{11}}$$ which completes the proof of the claim since $|\rho|_{H^{11}}\leq 4\epsilon$ by assumption. 

Now with the claim at hand we see, using (2.1):

$$|\rho|_{H^5} \leq \frac{C\epsilon}{(t+1)^{1/4}}+\int_0^t \frac{C\sqrt{\epsilon}}{(t-s+1)^{1/4}} |u_2|_{H^{10}}(s) \sqrt{|\rho|_{H^5}(s)} ds.$$

Now let's estimate $|u_2|_{H^{10}}$ using estimate (2.6) which has no loss of derivatives, again, using the Duhamel formula:

$$|u_2|_{H^{10}} \leq \frac{C\epsilon}{(t+1)}+ \int_0^t \frac{|u\cdot\nabla\rho|_{H^{10}}}{(t-s+1)}$$

We will need the following estimate:

\emph{Claim:}

$$|u\cdot\nabla\rho|_{H^{10}} \leq C\sqrt{\epsilon}|u_2|_{H^{10}} \sqrt{|\rho|_{H^5}}$$

Indeed, as before, $$|u\cdot\nabla\rho|_{H^{10}} \leq C\Big (  |u_1|_{H^{10}} |\partial_x\rho|_{L^\infty}+|u_1|_{L^\infty} |\partial_x\rho|_{H^{10}}+|u_2|_{H^{10}} |\partial_y\rho|_{L^\infty}+|u_2|_{L^\infty} |\partial_y\rho|_{H^{10}}\Big) $$

However, $$|u_1|_{H^{10}}\leq |u_2|_{H^{10}}^{1/2}|\rho|_{H^{10}}^{1/2}\leq |u_2|_{H^{10}}^{1/2}\sqrt{\epsilon},$$
$$|\partial_x\rho|_{L^\infty} \leq |u_1|_{H^3}\leq  |u_2|_{H^3}^{1/2}|\rho|_{H^3}^{1/2},$$

$$|u_1|_{L^\infty}\leq |u_1|_{H^{2}}\leq |u_2|_{H^2}^{1/2} |\rho|_{H^2}^{1/2},$$
$$|\partial_x\rho|_{H^{10}} \leq |u_1|_{H^{11}}\leq  |u_2|_{H^{10}}^{1/2}|\rho|_{H^{12}}^{1/2},$$

and

$$|\partial_y \rho|_{H^{10}}\leq |\rho|_{H^5}^{1/2} |\rho|_{H^{17}}^{1/2}.$$

This completes the proof of the claim.

\end{proof}

Hence, we have:

\begin{equation} |\rho|_{H^5} \leq \frac{C_{*}\epsilon}{(t+1)^{1/4}}+\int_0^t \frac{C\sqrt{\epsilon}}{(t-s+1)^{1/4}} |u_2|_{H^{10}}(s) \sqrt{|\rho|_{H^5}(s)} ds \end{equation}

and

\begin{equation} |u_2|_{H^{10}} \leq \frac{C_{*}\epsilon}{(t+1)}+\int_0^t \frac{C\sqrt{\epsilon}}{(t-s+1)} |u_2|_{H^{10}}(s) \sqrt{|\rho|_{H^5}(s)} ds\end{equation}

for some fixed $C_{*}.$

\vspace{4mm}

\emph{A Continuity Argument:}

\vspace{4mm}

Now, if we assume that $$|\rho|_{H^5}\leq 4\frac{C_{*}\epsilon}{(t+1)^{1/4}}$$ and  $$|u_2|_{H^5}\leq 4\frac{C_{*}\epsilon}{(t+1)}$$ on an interval $[0,T*]$ we will be able to apply inequality (2.8) to (3.5) and (3.6) to prove that, actually, $$|\rho|_{H^5}\leq 2\frac{C_{*}\epsilon}{(t+1)^{1/4}} $$ $$ |u_2|_{H^{10}}\leq  2\frac{C_{*}\epsilon}{(t+1)}$$ for all $t\in [0,T*]$ and,  by continuity, for all $t\in [0,T]$. This completes the proof of Proposition 3.3.

\subsection{Integrable Decay on $u_2$}

With Proposition 3.3 at hand, we can now proceed to prove that $u_2$ decays at an integrable rate. This will allow us to close the energy estimate (3.1) and finish the proof.

\begin{prop}

Let $\rho_0\in W^{5,1}$ with $|\rho_0|_{W^{5,1}}\leq \epsilon$ and assume that $|\rho|_{H^{20}}\leq 4\epsilon$ on [0,T].  Then, \begin{equation} |\nabla u_2|_{L^\infty} \lesssim \frac{\epsilon}{(1+t)^{5/4}},\end{equation}

for all $t\in [0,T]$.

\begin{proof}

Using linear estimate (2.3) and the Duhamel formula, we have:

$$|u_2|_{H^{2.5}} \lesssim \frac{\epsilon}{(t+1)^{5/4}} + \int_0^t \frac{|u\cdot\nabla\rho|_{W^{4,1}}(s)}{(t-s+1)^{5/4}}$$

$$\leq   \frac{\epsilon}{(t+1)^{5/4}} + \int_0^t \frac{ |u_1|_{H^4}(s)|\partial_x\rho|_{H^4}(s)+ |u_2|_{H^4}(s)|\partial_y|_{H^4}(s)}{(t-s+1)^{5/4}}ds$$

However, using Proposition 3.3, we have:

$$|u|_{H^{2.5}}\lesssim \frac{\epsilon}{(t+1)^{5/4}}+ \int_0^{t} \frac{\epsilon^2}{(t+s-1)^{5/4}(s+1)^{5/4}}ds.$$ Now we apply Lemma 2.6 and we have $$|u|_{H^{2.5}}\lesssim  \frac{\epsilon}{(t+1)^{5/4}}$$

\end{proof}

\end{prop}

\subsection{Finishing off the proof}

Using (3.7) and (3.1) we see that if $\epsilon$ is small enough, if we assume that $|\rho|_{H^{20}} \leq 4\epsilon$ on an interval of time $[0,T]$ while $|\rho_0|_{H^{20}}\leq \epsilon$, we actually have that $|\rho|_{H^{20}} \leq 2\epsilon$. This  implies that $|\rho|_{H^{20}}\leq 2\epsilon$ for all time and we are done.

\section{Asymptotic Stability on The Torus}

The proof of Theorem 1.4 is of a completely different nature when compared to the proof of Theorem 1.3. Indeed, in the $\mathbb{T}^2$ case, $\rho$ itself cannot decay. Indeed, if $\rho_0=g(y)$ then the solution is stationary. This causes a major difficulty in proving the global stability of $\rho(y)=y.$ Indeed we will see that there are two major difficulties:

(1) $\rho$ itself cannot decay.

(2) The linearized problem has (infinitely) many non-decaying modes. 

(3) Proving that $u$ decays fast enough requires a loss of derivatives. 

Note that a loss of derivatives in the decay estimate is not itself the problem-it is indeed the fact that $\rho$ itself cannot decay that makes derivative losses problematic. 

Indeed, as is outlined in the section on linear estimates below, it is possible to prove the following linear decay estimate:

\begin{equation} \big |e^{R_1^2t} R_1  \rho\big|_{L^2} \lesssim \frac{1}{\sqrt{t+1}} |\rho|_{L^2}\end{equation}

and 
 
\begin{equation} \big |e^{R_1^2t} R_1^2  \rho\big|_{L^2} \lesssim \frac{1}{{t+1}} |\rho|_{L^2}\end{equation}

This would imply, on a linear level, a decay on the order of $t^{-1}$ for $u_2$ which, when coupled with the energy inequality (3.2) would give us a shot at proving almost global existence. 

On the other hand, if one allows for an arbitrarily weak derivative loss, we can get integrable decay on $u_2.$

$$\big |e^{R_1^2}t R_1^2  \rho\big|_{L^2} \lesssim \frac{1}{{(t+1)^{1+\epsilon}}} |\rho|_{H^\epsilon}.$$

This derivative loss would seem to require that we close our estimates in a "super-smooth" space such as a Sobolev-Gevrey space (again, this is because $\rho$ will not decay). We, however, desire to prove a global stability result in Sobolev spaces. 

\subsection{Overcoming difficulties (1) and (2)}
Let's introduce some notation. Let $f:\mathbb{T}^2\rightarrow\mathbb{R}$ be a $C^1$ function. We define:

$$\tilde{f}:= \dashint_{-\pi}^{\pi} f(x,y)dx,$$
$$\bar{f}:= f-\tilde{f}.$$

Notice that $\tilde{f}$ is always a function of $y$ only. 

It is expected that $\bar{\rho}$ will decay and $\tilde{\rho}$ will just remain bounded. 

Now, $$\partial_t\rho+u\cdot\nabla\rho=R_1^2\rho.$$

Notice that $R_1^2\rho=R_1^2\bar\rho$ and $\bar u =u$. 

So we can write:

$$\partial_t \bar \rho +\overline{u\cdot\nabla\rho} =R_1^2\bar\rho$$

But $$\overline{u\cdot\nabla\rho}=\overline{u\cdot\nabla\bar\rho}+ \overline{u\cdot\nabla\tilde\rho}=\overline{u\cdot\nabla\bar\rho}+ \overline{u_2\partial_y\tilde\rho} = \overline{u\cdot\nabla\bar\rho}+ u_2\partial_y\tilde\rho $$
where the third equality is due to the fact that $\tilde\rho$ is a function of $y$ only and the fourth equality is due to the fact that $u_2 =\bar u_2$ and $\tilde\rho$ is a function of $y$ only. 

Hence, the equation for $\bar\rho$ reads:

$$ \partial_t \bar \rho +\overline{u\cdot\nabla\bar\rho}+ u_2\partial_y\tilde\rho =R_1^2\bar\rho. $$

But since $u_2=R_1^2\rho$ we have:

\begin{equation}  \partial_t \bar \rho +\overline{u\cdot\nabla\bar\rho} =R_1^2\bar\rho (1-\partial_y\tilde\rho) \end{equation}
and also the equation for $\tilde\rho:$

\begin{equation} \partial_t\tilde\rho + \widetilde{\partial_y(u_2\bar\rho)}=0\end{equation}

Our scheme for solving the problem will be as follows:

(A) Assume that $|\rho|_{H^{100}} \leq 4\epsilon.$

(B) Prove decay estimates on $$\partial_t f =(R_1^2 f) (1-G)$$ for a general small smooth function $G$. This may be called stability of the decay mechanism. 

(C) Bootstrap. 

\subsection{Estimates for the linearized problems on $\mathbb{T}^2$}

Our goal is to prove that if $\rho$ satisfies the linear problem:

$$\partial_t  \rho = R_1^2 \rho,$$ then $u=R_1 R^\perp \rho$ decays in time.  Using Fourier series, we may solve this equation exactly as:

$$\hat{\rho}(t,n) =e^{-\frac{n_1^2}{n_1^2+n_2^2}t} \hat{\rho_0}(n)$$

Now, it is clear that when $n_1=0$ there is no decay. However, the operator $R_1$ kills terms with $n_1=0$. Indeed,

$$R_1\hat{\rho}(t,n)=-i\frac{n_1}{|n|}e^{-\frac{n_1^2}{n_1^2+n_2^2}t} \hat{\rho_0}(n).$$

Hence, by Plancharel's theorem: $$|u|_{L^2}=|R_1 \rho|_{L^2}\leq \sum_{n_1\not=0}e^{-\frac{n_1^2}{n_1^2+n_2^2}t} |\hat{\rho_0}(n)|^2\leq \sum_{n_1\not=0}e^{-\frac{1}{n_1^2+n_2^2}t} |\hat{\rho_0}(n)|^2$$
Now we write:

$$\sum_{n_1\not=0}e^{-\frac{1}{n_1^2+n_2^2}t} |\hat{\rho_0}(n)|^2\leq e^{-\frac{t}{K^2}}|\rho_0|_{L^2} +\sum_{|n|\geq K}|\hat{\rho}_0(n)|^2$$

$$\leq e^{-\frac{t}{K^2}} |\rho_0|_{L^2}+ K^{-2s}|\rho_0|_{H^s}^2$$

Now, we want $u$ to decay, say in $H^{2+}$, like $t^{-2-\delta}$ for some $\delta>0$. This will be achieved if we take $s>2$ and $K=t^{0.5-\epsilon}$ for some $\epsilon>0$. Therefore, for each $\epsilon>0$, \begin{equation} |u|_{H^{2+\epsilon}}\leq \frac{1}{t^{2+\delta}} |\rho_0|_{H^{4+2\epsilon}}.\end{equation}

\subsection{Linear decay on the Torus in more generality}

Consider the following linear equation 

\begin{equation} \partial_t \rho=R_1^2\rho (1-G(y,t)). \end{equation}

We wish to prove linear decay estimates for this equation assuming that $G$ is sufficiently small and the initial data $\rho_0$ is such that $\tilde\rho_0:= \dashint \rho_0(x,y)dx\equiv 0.$

\begin{prop}
There exists $\delta>0$ such that if $|G|_{W^{11,\infty}}\leq \delta$ for all time, then, if $\rho(t)$ denotes the solution of equation (2.11) with initial data such that $\tilde\rho_0\equiv 0$, then 

\begin{equation} |\rho(t)|_{H^8} \lesssim \frac{|\rho_0|_{H^{10}}}{(1+t)^{5/2}},\,\,\,\, \forall\,\,\, t\geq 0. \end{equation}

\end{prop}

\begin{rema} The proof of this proposition is not as trivial as its counterpart when $G\equiv 0$. Indeed, because of the presence of the term $G(y,t)$ in (2.11), we cannot extract an exact formula for the solution because the $G$ term mixes the effect of all the Fourier coefficients while the operator $R_1^2$ is a Fourier multiplier. 
\end{rema}

\begin{proof}

First note $$\partial_t \int_{-\pi}^{\pi}  \rho(t,x,y)dx= \int_{-\pi}^{\pi} R_{1}^2 \rho (1-G(y))dy= \int_{-\pi}^{\pi} \partial_x(-\Delta)^{-1}\partial_x\rho(1-G(y))dx\equiv 0,$$ hence, if $\tilde\rho_0\equiv 0$ then $\tilde\rho(t)\equiv 0$.

Upon multiplying (2.11) by $\rho$ and integrating we see:

$$\partial_t |\rho|_{L^2}^2 = 2 \int R_1^2\rho (1-G) \rho.$$

Since $\tilde\rho\equiv 0$ we can write $\rho=\Delta \psi$. And we get:

$$\partial_t |\rho|_{L^2}^2 = -2 \int \partial_{xx}\psi (1-G) \Delta\psi=2 \int \partial_{x}\psi (1-G) \partial_x\Delta\psi.$$

$$= -2\int \nabla\big(\partial_x\psi(1-G)\big) \cdot\nabla\partial_x\psi= -2 \int |\nabla\partial_x\psi|^2 (1-G)+ \int \partial_x\psi \nabla G \cdot\nabla\partial_x\psi.$$

Now, assuming that $|G|_{H^1}\leq\delta<<1$ and applying the Poincar\'e inequality we get:

$$\partial_t |\rho|_{L^2}^2\leq -\int |\nabla\partial_x\psi|^2=-\int |R_1\psi|^2.$$

This yields that $|\rho|_{L^2}$ is bounded by its initial data. 

In fact, due to the fact that the Laplacian has discrete spectrum on $\mathbb{T}^2$ we can actually deduce that $\rho$ decays in $L^2$ so long as its higher derivatives are controlled. Indeed, 

$$\partial_t |\rho|_{L^2}^2\leq -\sum_{n,k} \frac{n^2}{n^2+k^2} |\rho_{n,k}|^2 \leq - \sum_{n,k} \frac{1}{n^2+k^2} |\rho_{n,k}|^2$$
$$\leq -\frac{1}{N} |\rho|_{L^2}^2+ \sum_{n^2+k^2>N}(\frac{1}{N}-\frac{1}{n^2+k^2})|\rho_{n,k}|^2\leq \frac{1}{N} \Big(-|\rho|_{L^2}^2+ \sum_{n^2+k^2>N}|\rho_{n,k}|^2\Big)$$

$$\leq -\frac{1}{N} |\rho|_{L^2}^2+ \frac{1}{N^5}\sum_{n^2+k^2>N} (n^2+k^2)^2 |\rho_{n,k}|^2\leq -\frac{1}{N} |\rho|_{L^2}^2+\frac{|\rho|_{H^2}}{N^5}.$$

Now take $N=\sqrt{t+1}$. 

This gives: 

$$\partial_t  |\rho|_{L^2}^2 \leq -\frac{|\rho|_{L^2}^2}{\sqrt{t+1}}  +\frac{|\rho|_{L^\infty_tH^2_x}^2}{(t+1)^{5/2}}.$$

\begin{lemm} 

Let $f$ be a positive $C^1$ function of $t$. 

Suppose that $$\partial_t f\leq -\frac{f}{\sqrt{t+1}}+ \frac{A}{(t+1)^{5/2}}$$ 
 
for some $A>0$.

Then,

$$f(t) \leq \frac{f(0)+A}{(t+1)^{5/2}}$$

\end{lemm}

\begin{proof} 

$$\partial_t (e^{2\sqrt{t+1}}f) \leq \frac{Ae^{2\sqrt{t+1}}}{(t+1)^{5/2}}$$

So, $$f(t)\leq e^{-2\sqrt{t+1}}f(0)+ \int_{0}^t \frac{A e^{2(\sqrt{s+1}-\sqrt{t+1})}}{(s+1)^{5/2}}ds.$$

The Lemma follows after we split the integral into two pieces: from $0$ to $t$ and $t/2$ to $t$. The integral from $0$ to $t/2$ decays exponentially. The second part of the integral decays like $(t+1)^{-5/2}$ multiplied by the factor:

$$\int_{t/2}^{t} e^{2(\sqrt{s+1}-\sqrt{t+1})} ds=\int_0^{\sqrt{t+1}-\sqrt{t/2+1}}  2\tau e^{-\tau}d\tau<C.$$

This completes the proof of the Lemma.

\end{proof}

Now applying Lemma 2.8 we see that $$|\rho|_{L^2} \leq \frac{|\rho|_{L^\infty([0,t]; H^2)}}{(t+1)^{5/2}}.$$

The idea is then to show that $|\rho(t)|_{H^2}\leq |\rho_0|$ and then this would give (2.12) with $H^8$ replaced by $L^2$ and $H^{10}$ replaced by $L^2$. We won't show this step as it will be clear from the $H^8$ estimate. 

Now we wish to prove a similar decay estimate for the higher derivatives.

First we will prove $|e^{Lt} \rho|_{H^{10}} \leq |\rho_0|_{H^{10}}.$ Define $J:= (-\Delta+1).$ Indeed, 

$$\partial_t \frac{1}{2}|\rho|_{H^{10}}^2 = \sum_{|s|\leq 10}\int\partial^s\Big(R_1^2\rho(1-G) \Big)\partial^s\rho=\sum_{|s|\leq 10}\int \partial^s \Big(\partial_{xx} \psi(1-G)\Big) \Delta \partial^s\psi$$
with $\Delta\psi=\rho$ as above. 

$$\partial_t |\rho|_{H^{10}}^2  =\sum_{|s| \leq 10}\int \partial^s \Big(\partial_{x} \psi(1-G)\Big) \Delta \partial^s\partial_x\psi $$

$$=\sum_{|s|\leq 10}\int (1-G)\partial^s\partial_{x} \psi \Delta \partial^s\partial_x\psi+\sum_{i=1}^{|s|} c_{i,s} \int \partial^{s-i}\partial_{x} \psi\partial^i G \Delta \partial^s\partial_x\psi$$

$$=-\sum_{|s|\leq 10 }\int (1-G)|\partial^s\partial_{x} \nabla\psi|^2+ \int\partial^s \partial_x \psi \nabla G \cdot  \nabla \partial^s\partial_x\psi +\sum_{i=1}^{|s|} c_{i,s} \int \nabla \Big(\partial^{s-i}\partial_{x} \psi\partial^i G\Big) \cdot\nabla \partial^s\partial_x\psi$$

$$\leq-\frac{3}{4} |R_1 \rho|_{H^{10}}^2+ C|G|_{W^{11,\infty}}|R_1\rho|_{H^{10}}^2$$

now if $|G|_{W^{11,\infty}}$ is small enough we see:

$$\partial_t |\rho|_{H^{10}}^2 \leq - |R_1\rho|_{H^{10}}^2$$ which implies that $|\rho|_{H^{10}}$ is uniformly bounded by its initial value:

$$|\rho|_{H^{10}}^2 \leq |\rho_0|_{H^{10}}^2.$$

By the same token, $$\partial_t |\rho|_{H^{8}}^2 \leq - |R_1\rho|_{H^{8}}^2.$$

Arguing as we did above when we proved the $L^2$ decay, we get:

$$|\rho|_{H^8} \leq\frac{|\rho_0|_{H^{10}}}{(1+t)^{5/2}}.$$

This concludes the proof of Proposition 2.6. 
\end{proof}
\section{The Proof of Theorem 1.2}

Let $\rho_0\in H^{20}$ be such that $|\rho_0|_{H^{20}}\leq \epsilon$ and suppose that $|\rho(t)|_{H^{20}}\leq 4\epsilon$ on a time interval $[0,T]$.

Then, the equation for $\bar\rho$ (4.3) reads:

$$\partial_t\bar\rho +u\cdot\nabla\bar\rho = L\bar\rho$$

with $$L\bar\rho= R_1^2\bar\rho (1-\partial_y\tilde\rho).$$ By assumption, $\tilde\rho$ is small in $H^{19}.$ This implies that $L$ has nice decay properties.

Using Duhamel's principle we have:

$$\bar\rho(t)= e^{Lt}\rho_0+ \int_0^{t} e^{L(t-s)} \overline {u\cdot\nabla\bar\rho}(s)ds $$

By the linear decay estimates on $L$ (Proposition 2.6) we know that  that $$|\bar\rho|_{H^{10}} \lesssim \frac{\epsilon}{(t+1)^2}+ \int_0^{t} \frac{1}{(t-s+1)^2} |u\cdot\nabla\bar\rho|_{H^{12}}(s)ds$$

$$\lesssim \frac{\epsilon}{(1+t)^2}+ \int_0^t \frac{1}{(t-s+1)^2} |\bar\rho|_{H^{10}}(s) |\rho|_{H^{13}}.$$

So,
$$|\bar{\rho}|_{H^{10}}(t)\lesssim \frac{\epsilon}{(1+t)^2}+ \int_0^t \frac{\epsilon}{(t-s+1)^2} |\bar\rho|_{H^{10}}(s)ds.$$

A simple bootstrap gives us that $$|\bar\rho|_{H^{10}}(t)\lesssim \frac{\epsilon}{(1+t)^2}.$$

Now using the energy estimate (3.1) we are finished.

\section{Acknowledgements}

The author acknowledges the support of an NSF Postdoctoral Research Fellowship. He also acknowledges helpful conversations with Jacob Bedrossian, Pierre Germain, and Nader Masmoudi.

\section{Appendix}

\subsection{Sharpness of the linear estimates}
\begin{prop} 
Estimates (2.4)-(2.6) are sharp in the sense that there exist Schwartz functions $e^{R_1^2t }f_i$ decays in $L^2$ just as dictated in the inequalities.
\end{prop}

\begin{proof}

We only give the proof for (2.4) and (2.6), the others being similar. 

For (2.4) we can take any radial function $f$. 
Recall that the Fourier transform of a radial function is radial. 
Then we have:

$$|e^{R_1^2t}f|_{L^2}^2= \int_0^{2\pi}\int e^{-2cos^2(\theta)t} |\hat{f}(r)|^2rdrd\theta= |f|_{L^2}^2 \int_0^{2\pi} e^{-cos^2(\theta)t}d\theta.$$

Hence, for any radial function $f$, $$|e^{R_1^2 t}f|_{L^2} \approx (1+t)^{-\frac{1}{4}} |f|_{L^2}.$$

This shows that (2.4) is sharp. 

Showing that (2.7) is sharp requires that we construct a sequence of functions which more and more (in Fourier space) along $\xi_1=0$. Indeed, by the dominated convergence theorem $|e^{R_1^2t} f|_{L^2}^2\rightarrow 0 $ as $t\rightarrow \infty$ for any $f\in L^2.$ The point is to take a sequence of $f's$ depending on $t$ for which $|f_t|_{L^2}^2=1$ and $|e^{R_1^2t} f_t|_{L^2}\not\rightarrow 0.$

Since we are working in $L^2$ we can look purely in Fourier space and we define $\phi(\xi_1,\xi_2)= \psi_t(\theta) g(r).$

Then we take $\phi=\hat{f}.$ Note that $f\in L^2$ if and only if $rg\in L^2(0,\infty)$ and $\psi_t\in L^2(0,2\pi)$. 

$$|e^{R_1^2} f|_{L^2}^2 = C_g \int_{0}^{2\pi} e^{-cos^2(\theta)t} |\psi_t(\theta)|^2d\theta. $$

Now we take $$\psi_t=\sqrt{t+1} \chi_{[\pi/2-\frac{1}{t+1},\pi/2+\frac{1}{t+1}]}.$$

Then we get:

$$|e^{R_1^2} f_t|_{L^2}^2 = C_f (t+1)\int_{\pi/2-\frac{1}{t+1}}^{\pi/2+\frac{1}{t+1}} e^{-cos^2(\theta)t}\geq . C_g (t+1) \int )\int_{\pi/2-\frac{1}{t+1}}^{\pi/2+\frac{1}{t+1}}e^{-1}d\theta \geq c. $$

This implies that $$|e^{R_1^2t}f|_{L^2\rightarrow L^2}\geq c$$ where $c$ is independent of $t$.

\end{proof}

\subsection{The necessity of having a nonlinearity with a "null structure"}

\begin{lemm}

There exists a smooth function $f$ on $\mathbb{T}^2$ such that for all $\epsilon>0$, the solution of 

$$\partial_t \rho + \rho\partial_y \rho = R_1^2 \rho. $$
 $$\rho_0 =\epsilon f$$ blows up at time $T\approx\frac{1}{\epsilon}$. 

\end{lemm}

\begin{rema}
This example illustrates that in certain non-linear problems, the damping operator $R_1^2$ has little or no ability to stop (or even delay!) the non-linearity from producing growth even for arbitrarily small data. 
\end{rema}

This is simply due to the fact that $f$ can be taken to be a function of $y$ since the linear term will vanish on functions of $y$ only. It is known that any non-trivial periodic function of $y$ develops a shock in finite time under the evolution of $\partial_t\rho +\rho\partial_y \rho=0.$

We believe that the same equation can also blow up in finite time on the whole space with arbitrarily small initial data.

\bibliographystyle{plain}

\end{document}